\documentclass[]{interact}

\usepackage{epstopdf}
\usepackage[caption=false]{subfig}

\usepackage[numbers,sort&compress]{natbib}
\bibpunct[, ]{[}{]}{,}{n}{,}{,}
\renewcommand\bibfont{\fontsize{10}{12}\selectfont}

\theoremstyle{plain}

\theoremstyle{definition}

\theoremstyle{remark}

\usepackage{graphicx}

\usepackage[utf8]{inputenc} 
\usepackage[T1]{fontenc}    
\usepackage{lmodern}
\usepackage{hyperref}       
\usepackage{url}            
\usepackage{booktabs}       
\usepackage{amsfonts}       
\usepackage{nicefrac}       
\usepackage{microtype}      

\usepackage{epsfig}
\usepackage{amssymb}
\usepackage{amsmath}
\usepackage{amsfonts}
\usepackage{bbding}
\usepackage{array}
\usepackage{caption,tabularx,booktabs}
\usepackage{arydshln}
\usepackage{xargs}                      
\usepackage[pdftex,dvipsnames]{xcolor}  

\usepackage{mathtools}  
\usepackage{amsmath}
\usepackage{amssymb}
\usepackage{tabulary}

\usepackage{algorithm}
\usepackage{algorithmic}
\usepackage{float}
\floatstyle{plain} 
\restylefloat{figure}

\usepackage{hyperref}
\usepackage{breakurl}

\newtheorem{thm}{Theorem}
\newtheorem{lem}{Lemma}
\newtheorem{prop}{Proposition}
\newtheorem{cor}{Corollary}

\newtheorem{ass}{Assumption}

\makeatletter
\newcounter{subthm} 
\let\savedc@thm\c@hyp

\newcommand{\normhyp}{%
  \let\c@hyp\savedc@hyp 
  \renewcommand\thehyp{\arabic{hyp}}%
} 
\makeatother

\makeatletter
\newcounter{subass} 
\let\savedc@ass\c@hyp

\makeatother

\newcommand\tagthis{\addtocounter{equation}{1}\tag{\theequation}}

\DeclareMathOperator{\Ocal}{\mathcal{O}} 
\newcommand{\eqdef}{\stackrel{\text{def}}{=}}

\renewcommand{\top}{T}

\usepackage{fancyhdr}

\usepackage{array}
\usepackage{caption,tabularx,booktabs}
\usepackage{arydshln}

\usepackage{xargs}                      
\usepackage[pdftex,dvipsnames]{xcolor}  
\newcommandx{\unsure}[2][1=]{\todo[inline,linecolor=red,backgroundcolor=red!25,bordercolor=red,#1]{#2}}
\newcommandx{\change}[2][1=]{\todo[linecolor=blue,backgroundcolor=blue!25,bordercolor=blue,#1]{#2}}
\newcommandx{\info}[2][1=]{\todo[linecolor=OliveGreen,inline,backgroundcolor=OliveGreen!25,bordercolor=OliveGreen,#1]{#2}}
\newcommandx{\improvement}[2][1=]{\todo[linecolor=Plum,inline,backgroundcolor=Plum!25,bordercolor=Plum,#1]{#2}}

\begin{document}


\title{Inexact SARAH Algorithm for Stochastic Optimization}

\author{
\name{Lam M. Nguyen\textsuperscript{a}\thanks{CONTACT Lam M. Nguyen. Email: LamNguyen.MLTD@ibm.com}, Katya Scheinberg\textsuperscript{b}, and Martin Tak\'{a}\v{c}\textsuperscript{c}}
\affil{\textsuperscript{a}IBM Research, Thomas J. Watson Research Center, Yorktown Heights, NY, USA; \textsuperscript{b}School of Operations Research and Information Engineering, Cornell University, Ithaca, NY, USA; \textsuperscript{c}Department of Industrial and Systems Engineering, Lehigh University, Bethlehem, PA, USA}
}

\maketitle

\begin{abstract}
We develop and analyze a variant of the SARAH algorithm, which does not require computation of the exact gradient. Thus this new method can be applied to general expectation minimization  problems rather than only finite sum problems. While the original SARAH algorithm, as well as its predecessor, SVRG, require an exact gradient computation on each outer iteration, the inexact variant of SARAH (iSARAH), which we develop here,  requires only stochastic gradient computed  on a mini-batch of sufficient size. The proposed method combines variance reduction via sample size selection and iterative stochastic gradient updates. We analyze the convergence rate of the algorithms for strongly convex and non-strongly convex cases,  under smooth assumption with appropriate mini-batch size selected for each case. We show that with an additional, reasonable, assumption iSARAH achieves the best known  complexity among stochastic methods in the case of non-strongly convex stochastic functions. 
\end{abstract}


\section{Introduction}\label{intro}

We consider the problem of stochastic  optimization 
\begin{align*}
\min_{w \in \mathbb{R}^d} \left\{ F(w) = \mathbb{E} [ f(w;\xi) ] \right\}, \tagthis \label{main_prob_expected_risk}  
\end{align*}
where $\xi$ is a random variable and $f$ has a Lipschitz continuous gradient. One of the most popular  applications of this problem is expected risk minimization in supervised learning. 
In this case, random variable $\xi$ represents a random data sample $(x,y)$, or a set of such samples $\{(x_i,y_i)\}_{i \in I}$.  
 We can consider a set of realizations $\{\xi_{[i]}\}_{i=1}^n$ of $\xi$ corresponding to a set of random samples $\{(x_i,y_i)\}_{i=1}^n$, and define $f_i(w) := f(w;\xi_{[i]})$. Then the sample average approximation of $F(w)$,  known as  empirical risk in supervised learning, is written as
\begin{align*}
\min_{w \in \mathbb{R}^d} \left\{ F(w) = \frac{1}{n}
\sum_{i=1}^n f_i(w) \right\}. \tagthis \label{main_prob_empirical_risk}  
\end{align*}

Throughout the paper, we assume  the existence of unbiased gradient estimator, that is $\mathbb{E}[\nabla f(w;\xi)] = \nabla F(w)$ for any fixed $w \in \mathbb{R}^d$. In addition we assume that there exists 
a lower bound of function $F$. 

In recent years, a class of variance reduction methods \cite{schmidt2017minimizing,SAG,SAGA,SVRG,S2GD,nguyen2017sarah} has been proposed for  problem \eqref{main_prob_empirical_risk} which have smaller  computational complexity than both, the full gradient descent method and the stochastic gradient method. All these methods rely on the finite sum form of \eqref{main_prob_empirical_risk} and are, thus, not readily extendable to  \eqref{main_prob_expected_risk}. In particular, SVRG  \cite{SVRG} and SARAH \cite{nguyen2017sarah} are two similar methods that consist of an outer loop, which includes one exact 
gradient  computation at each outer iteration and an inner loop with multiple   iterative stochastic gradient updates. The only difference between
SVRG and SARAH is how the iterative updates are performed in the inner loop. The advantage of SARAH is that the inner loop itself results in a convergent stochastic  gradient algorithm. Hence, it is possible to apply only one-loop of SARAH with sufficiently large number of steps to obtain an approximately optimal solution (in expectation). The convergence behavior of one-loop SARAH is similar to that of the standard stochastic gradient method \cite{nguyen2017sarah}. The multiple-loop SARAH algorithm matches convergence rates of SVRG in the strongly convex case, however, due to its convergent inner loop, it has an additional practical advantage of being able to use an adaptive inner loop size, as was demonstrated in \cite{nguyen2017sarah} for details). 

A version of SVRG algorithm, SCSG, has been recently proposed and analyzed in \cite{LihuaConvex,LihuaNonconvex}. 
While this method has been developed for 
\eqref{main_prob_empirical_risk} it can be directly applied to   \eqref{main_prob_expected_risk} because 
the exact gradient computation is replaced with a mini-batch stochastic gradient. The size of the  inner loop of SCSG is then set to a geometrically distributed random variable with distribution dependent on the size of the mini-batch used in the outer iteration. 
In this paper, we propose and analyze an inexact version of SARAH (iSARAH) which can be applied to solve
\eqref{main_prob_expected_risk}. Instead of exact gradient computation, a mini-batch gradient is computed using
a sufficiently large sample size. We develop total sample complexity analysis for this method under various convexity assumptions on $F(w)$. These complexity  results
are summarized  in Tables \ref{table_stronglyconvex}-\ref{table_generalconvex} and are compared to the result for  SCSG from \cite{LihuaConvex,LihuaNonconvex} when applied to \eqref{main_prob_expected_risk}. We also list the complexity bounds for SVRG, SARAH and SCSG when applied to finite sum problem \eqref{main_prob_empirical_risk}. 

As SVRG, SCSG and SARAH, iSARAH algorithm consists of the {\em outer loop}, which performs variance reduction by computing sufficiently accurate gradient estimate, and the {\em inner loop}, which performs the stochastic gradient updates. 
If only one outer iteration of SARAH is performed and then followed by  sufficiently many  inner iterations, we refer to this algorithm as one-loop SARAH.  In  \cite{nguyen2017sarah} one-loop SARAH is analyzed  and shown to match the complexity of stochastic gradient descent. Here along with multiple-loop iSARAH we analyze one-loop iSARAH, which is obtained from one-loop SARAH by replacing the first full gradient computation with a stochastic gradient based on a sufficiently large mini-batch. 

All our complexity results present the bound on the number of iterations it takes to achieve
$ \mathbb{E}[\|\nabla F(w)\|^2]\leq \epsilon$. These complexity results   are developed under the assumption that $f(w,\xi)$ is $L$-smooth for every realization of the random variable $\xi$. Table \ref{table_stronglyconvex} shows the complexity results in the case when $F(w)$, but not necessarily every realization $f(w,\xi)$, is $\mu$-strongly convex,  with $\kappa=L/\mu$ denoting the condition number.  Notice that results for one-loop iSARAH (and SARAH) are the same in the strongly convex case as in the non-strongly convex case. The convergence rate of the multiple-loop iSARAH, on the other hand, 
is better, in terms of dependence on $\kappa$, than the rate achieved by SCSG, which is the only other variance reduction 
method of the type we consider here,  that applies to  \eqref{main_prob_expected_risk}.  For example, if a small strongly convex perturbation is added to a convex objective function, then $\kappa = \Ocal(\frac{1}{\epsilon})$, and the total complexity of SCSG is  $\mathcal{O}\left( \left(\frac{\kappa}{\epsilon} + \kappa \right) \log\left( \frac{1}{\epsilon} \right) \right) = \mathcal{O}\left( \left( \frac{1}{\epsilon^2} \right) \log\left( \frac{1}{\epsilon} \right) \right)$ while iSARAH with multiple-loops has $\Ocal\left(  \max\left\{  \frac{1}{\epsilon}, \kappa \right\}  \log\left( \frac{1}{\epsilon} \right)   \right) = \mathcal{O}\left( \left( \frac{1}{\epsilon} \right) \log\left( \frac{1}{\epsilon} \right) \right)$. 
The non-strongly convex case is summarized in Table \ref{table_generalconvex}. In this case, the multiple-loop  iSARAH  achieves the best convergence rate with respect to $\epsilon$ among the compared stochastic methods, but this rate is derived under an  additional  assumption (Assumption \ref{ass_generalconvex_01} -- see \eqref{eq_ass_generalconvex_01}) which is discussed in detail in Section~\ref{sec_analysis}.  We note here that Assumption \ref{ass_generalconvex_01} is weaker than the, common in the analysis of stochastic algorithms, assumption that the iterates remain in a bounded set. In  recent paper \cite{paquette2018stochastic} a stochastic line search method is shown to have expected iteration complexity of $\Ocal\left( \frac{1}{\epsilon}  \right)$ under the assumption that the iterates remain in a bounded set, however, no total sample complexity is derived in \cite{paquette2018stochastic}. 

For non-strongly convex and non-convex problems,  we derive convergence rate for one-loop iSARAH with only requiring $\mathbb{E}[\| \nabla f(w;\xi) \|^2]$ finite at $w_*$ (non-strongly convex) and $w_0$ (non-convex). This convergence rate matches that of the general stochastic gradient algorithm, i.e. $\Ocal(\frac{1}{\epsilon^{2}})$, while the one-loop iSARAH can be viewed  as a type of stochastic gradient method with momentum. In order to derive the convergence results in the non-strongly convex and non-convex cases, the learning rate of one-loop iSARAH needs to be chosen in the order of $\Ocal(\epsilon)$, which could be practically undesirable. This is a limitation of one-loop stochastic gradient algorithms.

 Multiple-loop iSARAH method achieves the state-of-the-art complexity result of \\ $\Ocal\left( \frac{\sigma}{\epsilon^{3/2}}  + \frac{\sigma^2}{\epsilon} \right)$,  under bounded variance assumption, which requires that  $\mathbb{E}[\| \nabla f(w;\xi) - \nabla F(w) \|^2 ] \leq \sigma^2$, for some $\sigma > 0$ and $\forall w \in \mathbb{R}^d$. The analysis is a special case  of the analysis ProxSARAH \cite{Pham2019_ProxSARAH} and hence we do not derive it here. This results matches those for two new variants of SARAH, SPIDER and SpiderBoost, that have been proposed in \cite{fang2018spider,wang2018spiderboost}. SPIDER \cite{fang2018spider} uses the SARAH update rule as was originally proposed in \cite{nguyen2017sarah} and the mini-batch version of SARAH in \cite{NguyenLST17a}. SPIDER and SARAH are different in terms of updated iteration, which are $w_{t+1} = w_{t} - \eta (v_t / \|v_t\|)$ and $w_{t+1} = w_{t} - \eta v_t$, respectively. Also, SPIDER does not divide into outer loop and inner loop as SARAH does although SPIDER does also perform a full gradient update after a certain fixed number of iterations.
A recent technical report \cite{wang2018spiderboost} provides an improved version of SPIDER called SpiderBoost which allows a larger learning rate.

\begin{table}[h]
 \centering 
 \scriptsize
\caption{Comparison results (Strongly convex). Criteria: $\| \nabla F (w_{\mathcal{T}}) \|^2 \leq \epsilon $ (*), \\ $F (w_{\mathcal{T}}) - F(w_*) \leq \epsilon $ (**), and $\| w_{\mathcal{T}} -  w_* \|^2 \leq \epsilon $ (***). Assumption: $\mathbb{E}[\| \nabla f(w_*;\xi) \|^2] \leq \sigma_*^2$ (A). Note $\kappa = L/\mu$; ``Ad. Asm.'' is Additional Assumption}
\label{table_stronglyconvex}
\begin{tabular}{|c|c|c|c|}
\hline 
\textbf{Method} & \textbf{Bound} & \textbf{Problem} & \textbf{Ad. Asm.}  \\
\hline
\hline
\textbf{SARAH (one-loop)}  & N/A  & N/A & N/A  \\
\hline
\textbf{SARAH (multiple-loop)}  & $\mathcal{O}\left( (n + \kappa) \log\left( \frac{1}{\epsilon} \right) \right)$  & \eqref{main_prob_empirical_risk} (*) & None     \\
\hline
SVRG  & $\mathcal{O}\left( (n + \kappa) \log\left( \frac{1}{\epsilon} \right) \right)$  & \eqref{main_prob_empirical_risk} (*) & None   \\
\hline
SCSG  & $\mathcal{O}\left(  \left( \min\left\{ \frac{\sigma_*^2}{\mu \epsilon } , n \right\} + \kappa  \right) \log \left( \frac{\max\{ n , \kappa \}}{n}  \cdot \frac{1}{\epsilon} \right) \right)$   & \eqref{main_prob_empirical_risk} (**) & (A)   \\
\hline
\hline
SCSG & $\mathcal{O}\left(  \left( \frac{\sigma_*^2}{\mu \epsilon } + \kappa  \right) \log \left( \frac{1}{\epsilon} \right) \right)$  & \eqref{main_prob_expected_risk} (**) & (A)  \\
\hline
\textbf{iSARAH (one-loop)}  & N/A & N/A & N/A   \\
\hline
{\color{red}\textbf{iSARAH (multiple-loop)}} & $\textcolor{red}{\Ocal\left(  \max\left\{  \frac{\sigma_*^2}{\epsilon}, \kappa \right\} \log\left( \frac{1}{\epsilon} \right)   \right)}$  &  \eqref{main_prob_expected_risk} (*) & (A)    \\
\hline
\end{tabular}
\end{table}

\begin{table}[h]
 \centering 
 \scriptsize
\caption{Comparison results (Non-strongly convex).  Criteria: $\| \nabla F (w_{\mathcal{T}}) \|^2 \leq \epsilon $ (*), \\ $F (w_{\mathcal{T}}) - F(w_*) \leq \epsilon $ (**). Assumption: $\mathbb{E}[\| \nabla f(w_*;\xi) \|^2] \leq \sigma_*^2$ (A), $F(w) - F(w_*) \leq M \| \nabla F(w) \|^2 + N$ (B). ``Ad. Asm.'' is Additional Assumption}
\label{table_generalconvex}
\begin{tabular}{|c|c|c|c| }
\hline 
\textbf{Method} & \textbf{Bound} & \textbf{Problem} & \textbf{Ad. Asm.}  \\
\hline
\hline
\textbf{SARAH (one-loop)} & $\mathcal{O}\left( n + \frac{L}{\epsilon^2} \right)$  & \eqref{main_prob_empirical_risk} (*) & None  \\
\hline
\textbf{SARAH (multiple-loop)} & $\mathcal{O}\left( \left(n + \max\left\{  LM , \frac{L N}{\epsilon}  \right\} \right) \log\left( \frac{1}{\epsilon} \right) \right)$  & \eqref{main_prob_empirical_risk} (*) & (B)  \\
\hline
SVRG   & $\mathcal{O}\left( n + \frac{L \sqrt{n}}{\epsilon} \right)$   & \eqref{main_prob_empirical_risk} (*) & None  \\
\hline
SCSG  & $\mathcal{O}\left( \min \left \{  \frac{\sigma_*^2}{\epsilon^2} , \frac{n L}{\epsilon} \right \}  \right)$ & \eqref{main_prob_empirical_risk} (**) & (A) \\
\hline
\hline
SCSG &  $\mathcal{O}\left( \frac{\sigma_*^2}{\epsilon^2} \right)$  & \eqref{main_prob_expected_risk} (**) & (A)  \\
\hline
{\color{red}\textbf{iSARAH (one-loop)}}  & $\textcolor{red}{\Ocal \left( \frac{\max\{ L , \sigma_*^2 \}}{\epsilon} + \frac{\max\{ L^2 , \sigma_*^4 \}}{\epsilon^2} \right)}$ & \eqref{main_prob_expected_risk} (*) & (A) \\
\hline
{\color{red}\textbf{iSARAH (multiple-loop)}} & $\textcolor{red}{\Ocal\left(  \max\left\{ LM , \frac{\max\{ L N , \sigma_*^2 \}}{\epsilon}  \right \} \log \left( \frac{1}{\epsilon}  \right)  \right)}$  & \eqref{main_prob_expected_risk} (*) & (A) and (B)   \\
\hline
\end{tabular}
\end{table}

Note that, in the original SARAH paper, the result for non-strongly convex case is also derived under the assumption $F(w) - F(w_*) \leq M \| \nabla F(w) \|^2 + N$, for some $M > 0$ and $N > 0$, however this assumption was not explicitly stated. 

\section{The Algorithm}\label{sec_algorithm}

Like SVRG and SARAH, iSARAH consists of the outer loop and the inner loop. The inner loop performs recursive stochastic gradient updates, 
while the outer loop of SVRG and SARAH compute the exact gradient. Specifically, given an iterate $w_0$ at the beginning of each outer loop,    SVRG and SARAH compute $v_0=\nabla F (w_{0})$. 
The only difference between SARAH and iSARAH is that the latter replaces the exact gradient computation by a gradient estimate based on a sample set of size $b$. 

In other words, given a  iterate $w_0$ and a sample set size $b$,  $v_0$ is  a random vector computed as 
\begin{align*}
v_0 = \frac{1}{b} \sum_{i=1}^b \nabla f(w_0; \zeta_{i}), \tagthis \label{eq:v_0}
\end{align*}
where $\left\{\zeta_{i}\right\}_{i=1}^b$ are i.i.d.\footnote{Independent and identically distributed random variables. We note from probability theory that if $\zeta_1,\dots,\zeta_b$ are i.i.d. random variables then $g(\zeta_1),\dots,g(\zeta_b)$ are also i.i.d. random variables for any measurable function $g$.} and $\mathbb{E}[\nabla f(w_0;\zeta_{i}) | w_0] = \nabla F(w_0)$. We have 
\begin{align*}
    \mathbb{E}[ v_0 | w_0 ] = \frac{1}{b} \sum_{i=1}^b \nabla F(w_0)   = \nabla F(w_0).
\end{align*}
The larger $b$ is, the more accurately the gradient estimate $v_0$ approximates  $\nabla F(w_0)$. The key idea of the analysis of  iSARAH is to establish bounds on $b$ which ensure sufficient accuracy for recovering original SARAH convergence rate.

The key step of the algorithm is a recursive update of the stochastic gradient estimate \textit{(SARAH update)} 
\begin{equation}\label{eq:vt}
   v_{t} = \nabla f (w_{t}; \xi_t) - \nabla f(w_{t-1}; \xi_t) + v_{t-1},
\end{equation}
followed by the iterate update
\begin{equation}\label{eq:iterate}
w_{t+1} = w_{t} - \eta v_{t}.
\end{equation}

Let $\mathcal{F}_t = \sigma(w_0,w_1,\dots,w_t)$ be the
$\sigma$-algebra generated by $w_0,w_1,\dots,w_t$. We note that $\xi_t$ is independent of $\mathcal{F}_t$ and that $v_t$ is a  \textit{biased estimator} of the gradient $\nabla F (w_{t})$.
\begin{align*}
\mathbb{E}[ v_{t} | \mathcal{F}_t ] &= \nabla F (w_{t}) - \nabla F(w_{t-1}) + v_{t-1}. 
\end{align*}

In contrast, SVRG update is given by 
\begin{equation}\label{eq:svrgvt}
   v_{t} = \nabla f (w_{t}; \xi_t) - \nabla f(w_{0}; \xi_t) + v_{0} 
\end{equation}
which implies that  $v_t$ is an  \textit{unbiased estimator} of $\nabla F (w_{t})$. 

The outer loop of iSARAH algorithm is summarized in Algorithm \ref{isarah} and the inner loop is summarized in 
Algorithm \ref{isarah_in}. 

\begin{algorithm}[t]
   \caption{Inexact SARAH (iSARAH)}
   \label{isarah}
\begin{algorithmic}
   \STATE {\bfseries Parameters:} the learning rate $\eta > 0$ and the inner loop size $m$, the sample set size $b$.
   \STATE {\bfseries Initialize:} $\tilde{w}_0$.
   \STATE {\bfseries Iterate:}
   \FOR{$s=1,2,\dots, \mathcal{T}$,}
   \STATE $\tilde{w}_s = \text{iSARAH-IN}(\tilde{w}_{s-1},\eta,m, b)$.	
   \ENDFOR
   \STATE \textbf{Output:} $\tilde{w}_\mathcal{T}$.
\end{algorithmic}
\end{algorithm} 

\begin{algorithm}[t]
   \caption{iSARAH-IN$(w_0,\eta,m,b)$}
   \label{isarah_in}
\begin{algorithmic}
   \STATE {\bfseries Input:} $w_0 (= \tilde{w}_{s-1})$ the learning rate $\eta > 0$, the inner loop size $m$, the sample set size $b$.
   \STATE Generate random variables $\left\{\zeta_{i}\right\}_{i=1}^b$ i.i.d.
   \STATE Compute $v_0 = \frac{1}{b} \sum_{i=1}^b \nabla f(w_0; \zeta_{i})$.
   \STATE $w_1 = w_0 - \eta v_0$.
   \STATE {\bfseries Iterate:}
   \FOR{$t=1,\dots,m-1$,}
   \STATE Generate a random variable $\xi_t$
   \STATE $v_{t} = \nabla f (w_{t}; \xi_t) - \nabla f(w_{t-1}; \xi_t) + v_{t-1}$.
   \STATE $w_{t+1} = w_{t} - \eta v_{t}$.
   \ENDFOR
   \STATE Set $\tilde{w} = w_{t}$ with $t$ chosen uniformly at random from $\{0,1,\dots,m\}$
   \STATE \textbf{Output:} $\tilde{w}$
\end{algorithmic}
\end{algorithm} 

 As a variant of SARAH,  iSARAH inherits the special property that a one-loop iSARAH, which is the variant of  Algorithm~\ref{isarah} with $\mathcal{T}=1$ and $m\to \infty$, is a convergent algorithm. In the next section we provide the analysis for both one-loop and multiple-loop versions of i-SARAH. 

\textbf{Convergence criteria}. Our iteration complexity analysis aims to bound the total number of stochastic gradient evaluations needed to achieve a desired bound on the gradient norm. For that we will need to bound the number of outer iterations $\mathcal{T}$  which is needed to  guarantee that 
$\|\nabla F(w_\mathcal{T})\|^2\leq \epsilon$ and also to bound $m$ and $b$. Since the algorithm is stochastic and $w_\mathcal{T}$ is random
the  $\epsilon$-accurate solution is  only achieved in expectation. 
 i.e.,
\begin{equation}\label{eq:accuracy}
\mathbb{E} [\| \nabla F(w_\mathcal{T}) \|^2] \leq \epsilon.
\end{equation}

\section{Convergence Analysis of iSARAH}\label{sec_analysis}

\subsection{Basic Assumptions}

The  analysis of the proposed algorithm will be performed under apropriate subset of the following key assumptions. 
\begin{ass}[$L$-smooth]
\label{ass_Lsmooth}
\textit{$f(w;\xi)$ is $L$-smooth for every realization of $\xi$, i.e., there exists a constant $L > 0$ such that
\begin{align*}
\| \nabla f(w; \xi) - \nabla f(w'; \xi) \| \leq L \| w - w' \|, \ \forall w,w' \in \mathbb{R}^d. \tagthis \label{eq:Lsmooth_basic}
\end{align*}
}
\end{ass}
 Note that this assumption implies that $F(w) = \mathbb{E}[f(w; \xi)]$ is also \emph{L-smooth}. The following strong convexity assumption
 will be made for the appropriate parts of the analysis, otherwise, it would be dropped.  
\begin{ass}[$\mu$-strongly convex]
\label{ass_stronglyconvex}
\textit{The function $F: \mathbb{R}^d \to \mathbb{R}$, is $\mu$-strongly convex, i.e., there exists a constant $\mu > 0$ such that $\forall w,w' \in \mathbb{R}^d$, 
\begin{gather*}
F(w)   \geq  F(w') + \nabla F(w')^\top (w - w') + \tfrac{\mu}{2}\|w - w'\|^2.
\end{gather*}}
\end{ass}
Under Assumption \ref{ass_stronglyconvex}, let us define the (unique) optimal solution of \eqref{main_prob_empirical_risk} as $w_{*}$, 
Then strong convexity of $F$  implies that 
 \begin{equation}\label{eq:strongconvexity2}
 2\mu [ F(w) - F(w_{*})] \leq  \| \nabla F(w)\|^2, \ \forall w \in \mathbb{R}^d. 
 \end{equation}
 Under strong convexity assumption we will use $\kappa$ to denote  the condition number $\kappa\eqdef L/\mu$. 


Finally, as a special case of the strong convexity with $\mu=0$, we state the general convexity assumption, which we will use for some of the convergence results.
\begin{ass}[Convex]
\label{ass_convex}
\textit{$f(w;\xi)$ is convex for every realization of $\xi$, i.e., $\forall w,w' \in \mathbb{R}^d$, 
\begin{gather*}
f(w; \xi)    \geq f(w'; \xi) + \nabla f(w'; \xi)^\top (w - w'). 
\end{gather*}}
\end{ass}
We note that Assumption  \ref{ass_stronglyconvex} does not imply Assumption~\ref{ass_convex}, because the latter applies to all realizations, while the former applied only to the expectation.

Hence in our analysis, depending on the result we aim at, we will require Assumption~\ref{ass_convex} to hold by itself, or Assumption  \ref{ass_stronglyconvex}  and Assumption~\ref{ass_convex} to hold together.
 We will always use Assumption  \ref{ass_Lsmooth}.


We also assume that, the $\mathbb{E}[ \| \nabla f(w;\xi) \|^2$ at $w_*$ is bounded by some constant. 
\begin{ass}\label{ass_bounded_gradient_solution}
\textit{There exists some $\sigma_* > 0$ such that
\begin{align*}
    \mathbb{E}[ \| \nabla f(w_*;\xi) \|^2 \leq \sigma_*^2, \tagthis \label{eq_bounded_gradient_solution}
\end{align*}
where $w_{*}$ is any optimal solution of $F(w)$; and $\xi$ is some random variable.  }
\end{ass}

\subsection{Existing Results}

We provide some well-known results from the existing literature that support our theoretical analysis as follows. First, we start introducing two standard lemmas in smooth convex optimization (\cite{nesterov2004}) for a general function $f$. 
\begin{lem}[Theorem 2.1.5 in \cite{nesterov2004}]
\label{lem_tech_01}
\textit{Suppose that $f$ is convex and $L$-smooth. Then, for any $w$, $w' \in \mathbb{R}^d$, 
\begin{gather*}
f(w) \leq f(w') + \nabla f(w')^T(w-w') + \frac{L}{2}\|w-w'\|^2,
\tagthis\label{eq:Lsmooth}
\\
f(w) \geq f(w') + \nabla f(w')^\top (w - w') + \frac{1}{2L}\| \nabla f(w) - \nabla f(w') \|^2, 
\tagthis 
\\
(\nabla f(w) - \nabla f(w'))^\top (w - w') \geq \frac{1}{L}\| \nabla f(w) - \nabla f(w') \|^2. 
\tagthis\label{ineq_convex} 
\end{gather*}}
\end{lem}

Note that \eqref{eq:Lsmooth} does not require the convexity of $f$. 

\begin{lem}[Theorem 2.1.11 in \cite{nesterov2004}]
\label{lem_convex_lowerbound}
\textit{Suppose that $f$ is $\mu$-strongly convex and $L$-smooth. Then, for any $w$, $w' \in \mathbb{R}^d$, 
\begin{align*}
&(\nabla f(w) - \nabla f(w'))^\top (w - w') \geq \frac{\mu L}{\mu + L} \|w - w'\|^2  
 + \frac{1}{\mu + L}\| \nabla f(w) - \nabla f(w') \|^2. \tagthis\label{eqdasfsadfsa} 
\end{align*}}
\end{lem}

The following existing results are more specific properties of component functions $f(w;\xi)$. 

\begin{lem}[\cite{SVRG}]\label{lem_bound_diff_grad}
\textit{Suppose that Assumptions \ref{ass_Lsmooth} and \ref{ass_convex} hold. Then, $\forall w \in \mathbb{R}^d$, 
\begin{align*}
\mathbb{E} [ \| \nabla f(w;\xi) - \nabla f(w_{*};\xi) \|^2 ]  \leq 2L [ F(w) - F(w_{*})], \tagthis \label{eq_bound_diff_grad}
\end{align*}
where $w_{*}$ is any optimal solution of $F(w)$.}
\end{lem}

\begin{lem}[Lemma 1 in \cite{Nguyen2018sgd_hogwild}] \label{lem_bounded_secondmoment_04}
\textit{Suppose that Assumptions \ref{ass_Lsmooth} and \ref{ass_convex} hold. Then, for $\forall w \in \mathbb{R}^d$, 
\begin{gather*}
\mathbb{E} [ \|\nabla f(w;\xi)\|^2 ] \leq  4 L [ F(w) - F(w_{*}) ] + 2 \mathbb{E} [ \|\nabla f(w_{*};\xi)\|^2 ], \tagthis \label{eq_bounded_secondmoment_04} 
\end{gather*}
where $w_{*}$ is any optimal solution of $F(w)$.}
\end{lem}

\begin{lem}[Lemma 1 in \cite{Nguyen2018sgd_dnn}]\label{lem_bound_w0}
\textit{Let $\xi$ and $\{\xi_i\}_{i=1}^b$ be i.i.d. random variables with $\mathbb{E}[\nabla f(w;\xi_i )] = \nabla F(w)$, $i = 1,\dots,b$, for all $w \in \mathbb{R}^d$. Then, 
\begin{align*}
\mathbb{E} \left[ \left\| \frac{1}{b} \sum_{i=1}^b \nabla f(w; \xi_i)  - \nabla F(w) \right\|^2  \right] = \frac{\mathbb{E} [ \| \nabla f(w; \xi) \|^2 ] - \| \nabla F(w) \|^2 }{b}. \tagthis \label{eq_prob_bound_w0}
\end{align*}}
\end{lem}

The proof of this Lemma is in \cite{Nguyen2018sgd_dnn}.

Lemmas \ref{lem_bounded_secondmoment_04} and \ref{lem_bound_w0} clearly imply the following result. 
\begin{cor}\label{cor_two_lemmas}
\textit{Suppose that Assumptions \ref{ass_Lsmooth} and \ref{ass_convex} hold. Let $\xi$ and $\{\xi_i\}_{i=1}^b$ be i.i.d. random variables with $\mathbb{E}[\nabla f(w;\xi_i )] = \nabla F(w)$, $i = 1,\dots,b$, for all $w \in \mathbb{R}^d$. Then, 
\begin{align*}
& \mathbb{E} \left[ \left\| \frac{1}{b} \sum_{i=1}^b \nabla f(w; \xi_i)  - \nabla F(w) \right\|^2  \right] \\ & \qquad \qquad \leq \frac{4 L [ F(w) - F(w_{*}) ] + 2 \mathbb{E} [ \|\nabla f(w_{*};\xi)\|^2 ] - \| \nabla F(w) \|^2 }{b}, \tagthis \label{eq_prob_bound_w0_02}
\end{align*}
where $w_{*}$ is any optimal solution of $F(w)$. }
\end{cor} 

Based on the above lemmas, we will show in detail how to achieve our main results in the following subsections.   

\subsection{Special Property of SARAH Update}\label{sec_sarah_update}

The most important property of the SVRG algorithm is  the variance reduction of the steps. This property holds as the number of outer iteration grows, but it does not hold, if only the number of inner iterations increases.  In other words, if we simply run the inner loop for many iterations (without executing additional outer loops), the variance of the steps does not reduce in the case of SVRG, while it goes to zero in the case of SARAH with large learning rate in the strongly convex case. We recall the SARAH update as follows. 
\begin{equation}\label{eq:vt_update}
   v_{t} = \nabla f (w_{t}; \xi_t) - \nabla f(w_{t-1}; \xi_t) + v_{t-1},
\end{equation}
followed by the iterate update:
\begin{equation}
w_{t+1} = w_{t} - \eta v_{t}.
\end{equation}

We will now show that $\|v_t\|^2$ is going to zero in expectation in the \textit{strongly convex} case. These results substantiate our conclusion that SARAH uses more stable stochastic gradient estimates than SVRG. 

\begin{prop}\label{thm_bounded_moment_stronglyconvexP}
\textit{Suppose that Assumptions \ref{ass_Lsmooth}, \ref{ass_stronglyconvex} and \ref{ass_convex} hold. Consider $v_{t}$ defined by \eqref{eq:vt_update} with $\eta < 2/L$ and any given $v_0$. Then, for any $t\geq 1$,
\begin{align*}
\mathbb{E}[\|v_{t}\|^2]
 &\leq \left[ 1 - \left(\tfrac{2}{\eta L} - 1 \right) \mu^2 \eta^2  \right] \mathbb{E}[\|v_{t-1}\|^2]
 \\
 &\leq \left[ 1 - \left(\tfrac{2}{\eta L} - 1 \right) \mu^2 \eta^2  \right]^{t} \| v_0 \|^2.
\end{align*}}
\end{prop}

The proof of this Proposition can be derived directly from Theorem 1a in \cite{nguyen2017sarah}. This result implies that by choosing $\eta=\Ocal(1/L)$, we obtain the linear  convergence of $\|v_t\|^2$ in expectation with the rate $(1-1/\kappa^2)$. 

We will provide our convergence analysis in detail in next sub-section. We will divide our results into two parts:
the \textit{one-loop} results corresponding to iSARAH-IN (Algorithm \ref{isarah_in}) and the \textit{multiple-loop} results corresponding to iSARAH (Algorithm \ref{isarah}). 

\subsection{One-loop (iSARAH-IN) Results}

We begin with providing two useful lemmas that do not require convexity assumption. Lemma \ref{lem_main_derivation} bounds the sum of expected values of $\| \nabla F(w_{t})\|^2$; and Lemma \ref{lem:var_diff_01} expands the value of $\mathbb{E}[\| \nabla F(w_{t}) - v_{t} \|^2]$. 
\begin{lem}\label{lem_main_derivation}
\textit{Suppose that Assumption \ref{ass_Lsmooth} holds. Consider iSARAH-IN (Algorithm \ref{isarah_in}). Then, we have 
\begin{align*}
 \sum_{t=0}^{m} \mathbb{E}[ \| \nabla F(w_{t})\|^2 ] & \leq \frac{2}{\eta} \mathbb{E}[ F(w_{0}) - F(w_{*})] \\ & \qquad + \sum_{t=0}^{m} \mathbb{E}[ \| \nabla F(w_{t}) - v_{t} \|^2 ]  
 - ( 1 - L\eta ) \sum_{t=0}^{m} \mathbb{E} [ \| v_{t} \|^2 ], \tagthis \label{eq:001} 
\end{align*}
where $w_{*} = \arg \min_{w} F(w)$. }
\end{lem}

\begin{proof}
By Assumption \ref{ass_Lsmooth} and $w_{t+1} = w_{t} - \eta v_{t}$, we have
\begin{align*}
\mathbb{E}[ F(w_{t+1})] & \overset{\eqref{eq:Lsmooth}}{\leq}  \mathbb{E}[ F(w_{t})] - \eta \mathbb{E}[\nabla F(w_{t})^\top v_{t}] 
+ \frac{L\eta^2}{2} \mathbb{E} [ \| v_{t} \|^2 ] 
\\
& = \mathbb{E}[ F(w_{t})] - \frac{\eta}{2} \mathbb{E}[ \| \nabla F(w_{t})\|^2 ] 
+ \frac{\eta}{2} \mathbb{E}[ \| \nabla F(w_{t}) - v_{t} \|^2 ] 
\\ & \qquad - \left( \frac{\eta}{2} - \frac{L\eta^2}{2} \right) \mathbb{E} [ \| v_{t} \|^2 ],
\end{align*}
where the last equality follows from the fact
$a_1^T a_2 = \frac{1}{2}\left[\|a_1\|^2 + \|a_2\|^2 - \|a_1 - a_2\|^2\right].$

By summing over $t = 0,\dots,m$, we have
\begin{align*}
\mathbb{E}[ F(w_{m+1})] & \leq  \mathbb{E}[ F(w_{0})] - \frac{\eta}{2} \sum_{t=0}^{m} \mathbb{E}[ \| \nabla F(w_{t})\|^2 ] + \frac{\eta}{2} \sum_{t=0}^{m} \mathbb{E}[ \| \nabla F(w_{t}) - v_{t} \|^2 ] \\ & \qquad 
- \left( \frac{\eta}{2} - \frac{L\eta^2}{2} \right) \sum_{t=0}^{m} \mathbb{E} [ \| v_{t} \|^2 ],  
\end{align*}
which is equivalent to ($\eta>0$):
\allowdisplaybreaks
\begin{align*}
\sum_{t=0}^{m} \mathbb{E}[ \| \nabla F(w_{t})\|^2 ]  & \leq \frac{2}{\eta} \mathbb{E}[ F(w_{0}) - F(w_{m})] \\ & \qquad + \sum_{t=0}^{m} \mathbb{E}[ \| \nabla F(w_{t}) - v_{t} \|^2 ]  
 - ( 1 - L\eta ) \sum_{t=0}^{m} \mathbb{E} [ \| v_{t} \|^2 ] \\
& \leq \frac{2}{\eta} \mathbb{E}[ F(w_{0}) - F(w_{*})] \\ & \qquad + \sum_{t=0}^{m} \mathbb{E}[ \| \nabla F(w_{t}) - v_{t} \|^2 ]  
 - ( 1 - L\eta ) \sum_{t=0}^{m} \mathbb{E} [ \| v_{t} \|^2 ],    
\end{align*}

where the second inequality follows since $w_{*} = \arg \min_{w} F(w)$.
\end{proof}

\begin{lem}\label{lem:var_diff_01}
\textit{Suppose that Assumption \ref{ass_Lsmooth} holds. Consider $v_{t}$ defined by \eqref{eq:vt} in iSARAH-IN (Algorithm \ref{isarah_in}). Then for any $t\geq 1$, 
\begin{align*}
\mathbb{E}[ \| \nabla F(w_{t}) - v_{t} \|^2 ] 
& = \mathbb{E}[ \| \nabla F(w_{0}) - v_{0} \|^2 ] \\ & \qquad + \sum_{j = 1}^{t} \mathbb{E}[ \| v_{j} - v_{j-1} \|^2 ]  - \sum_{j = 1}^{t} \mathbb{E}[ \| \nabla F(w_{j}) - \nabla F(w_{j-1}) \|^2 ]. 
\end{align*}}
\end{lem}

\begin{proof}
Let $\mathcal{F}_j = \sigma(w_0,w_1,\dots,w_j)$ be the
$\sigma$-algebra generated by $w_0,w_1,\dots,w_j$\footnote{$\mathcal{F}_{j}$ contains all the information of $w_{0},\dots,w_{j}$ as well as $v_0,\dots,v_{j-1}$}. We note that $\xi_j$ is independent of $\mathcal{F}_j$. For $j \geq 1$, we have
\begin{align*}
& \mathbb{E}[ \| \nabla F(w_{j}) - v_{j} \|^2 | \mathcal{F}_{j} ] 
\\ & = \mathbb{E}[ \| [\nabla F(w_{j-1}) - v_{j-1} ] + [ \nabla F(w_{j}) - \nabla F(w_{j-1}) ]  - [ v_{j} - v_{j-1} ] \|^2 | \mathcal{F}_{j} ]
\\
& = \| \nabla F(w_{j-1}) - v_{j-1} \|^2 + \| \nabla F(w_{j}) - \nabla F(w_{j-1}) \|^2 + \mathbb{E} [ \| v_{j} - v_{j-1}  \|^2 | \mathcal{F}_{j} ] 
\\
&\quad + 2 ( \nabla F(w_{j-1}) - v_{j-1} )^\top ( \nabla F(w_{j}) - \nabla F(w_{j-1}) ) \\
&\quad - 2 ( \nabla F(w_{j-1}) - v_{j-1} )^\top \mathbb{E}[ v_{j} - v_{j-1} | \mathcal{F}_{j} ]  \\
&\quad - 2 ( \nabla F(w_{j}) - \nabla F(w_{j-1}) )^\top \mathbb{E}[ v_{j} - v_{j-1} | \mathcal{F}_{j} ] 
\\
& = \| \nabla F(w_{j-1}) - v_{j-1} \|^2 - \| \nabla F(w_{j}) - \nabla F(w_{j-1}) \|^2 + \mathbb{E} [ \| v_{j} - v_{j-1}  \|^2 | \mathcal{F}_{j} ], 
\end{align*}
where the last equality follows from
\begin{align*}
& \mathbb{E}[ v_{j} - v_{j-1} | \mathcal{F}_{j} ] \overset{\eqref{eq:vt}}{= }\mathbb{E}[ \nabla f(w_{j}; \xi_j) - \nabla f(w_{j-1}; \xi_j) | \mathcal{F}_{j} ] 
= \nabla F(w_{j}) - \nabla F(w_{j-1}).
\end{align*}

By taking expectation for the above equation, we have
\begin{align*}
\mathbb{E}[ \| \nabla F(w_{j}) - v_{j} \|^2 ] &= \mathbb{E}[ \| \nabla F(w_{j-1}) - v_{j-1} \|^2 ] \\ & \qquad - \mathbb{E}[ \| \nabla F(w_{j}) - \nabla F(w_{j-1}) \|^2 ] + \mathbb{E}[ \| v_{j} - v_{j-1} \|^2 ]. 
\end{align*}

By summing over $j = 1,\dots,t\ (t\geq 1)$, we have
\begin{align*}
\mathbb{E}[ \| \nabla F(w_{t}) - v_{t} \|^2 ] &= \mathbb{E}[\| \nabla F(w_{0}) - v_{0} \|^2] \\ & \qquad + \sum_{j = 1}^{t} \mathbb{E}[ \| v_{j} - v_{j-1} \|^2 ]  - \sum_{j = 1}^{t} \mathbb{E}[ \| \nabla F(w_{j}) - \nabla F(w_{j-1}) \|^2 ]. 
\end{align*}
\end{proof}

\subsubsection{Non-Strongly Convex Case}

In this subsection, we analyze one-loop results of Inexact SARAH (Algorithm \ref{isarah_in}) in the non-strongly convex case. We first derive the bound for $\mathbb{E}[ \| \nabla F(w_{t}) - v_{t} \|^2 ]$. 
\begin{lem}\label{lem_bound_var_diff_str_02}
\textit{Suppose that Assumptions \ref{ass_Lsmooth} and \ref{ass_convex} hold. Consider $v_{t}$ defined as \eqref{eq:vt} in SARAH (Algorithm \ref{isarah}) with $\eta < 2/L$. Then we have that for any $t\geq 1$, 
\begin{align*}
\mathbb{E}[ \| \nabla F(w_{t}) - v_{t} \|^2 ] 
&\leq  \frac{\eta L}{2 - \eta L} \Big[ \mathbb{E}[ \|v_{0} \|^2] - \mathbb{E}[\| v_{t} \|^2] \Big] + \mathbb{E}[\| \nabla F(w_{0}) - v_{0} \|^2]. \tagthis\label{eq:bound1}
\end{align*}}
\end{lem}

\begin{proof}
For $j \geq 1$, we have
\begin{align*}
& \mathbb{E}[\| v_{j} \|^2 | \mathcal{F}_{j}] \\
&= \mathbb{E}[\| v_{j-1} - (\nabla f(w_{j-1}; \xi_j) - \nabla f(w_{j}; \xi_j) \|^2 | \mathcal{F}_{j} ] 
\\ 
&= \|v_{j-1} \|^2 \\ & \quad + \mathbb{E}\Big[\| \nabla f(w_{j-1}; \xi_j) - \nabla f(w_{j}; \xi_j) \|^2  
 - \tfrac{2}{\eta}(\nabla f(w_{j-1}; \xi_j) - \nabla f(w_{j}; \xi_j))^\top (w_{j-1} - w_{j}) | \mathcal{F}_{j} \Big] 
 \\ 
& \overset{\eqref{ineq_convex}}{\leq} \|v_{j-1} \|^2 + \mathbb{E}\Big[\| \nabla f(w_{j-1}; \xi_j) - \nabla f(w_{j}; \xi_j) \|^2 
- \tfrac{2}{L \eta} \| \nabla f(w_{j-1}; \xi_j) - \nabla f(w_{j}; \xi_j) \|^2 | \mathcal{F}_{j} \Big] 
\\ 
& = \|v_{j-1} \|^2 + \left(1 - \tfrac{2}{\eta L}\right) \mathbb{E} [ \| \nabla f(w_{j-1}; \xi_j) - \nabla f(w_{j}; \xi_j) \|^2 | \mathcal{F}_{j} ] \\
& \overset{\eqref{eq:vt}}{= } \|v_{j-1} \|^2 + \left(1 - \tfrac{2}{\eta L}\right) \mathbb{E} [ \| v_{j} - v_{j-1} \|^2 | \mathcal{F}_{j} ],
\end{align*}
which, if we take expectation, implies that
$$
\mathbb{E}[\| v_{j} - v_{j-1} \|^2] 
\leq \frac{\eta L}{2 - \eta L} \Big[ \mathbb{E}[ \|v_{j-1} \|^2] - \mathbb{E}[\| v_{j} \|^2] \Big], 
$$
when $\eta < 2/{L}$.

\quad By summing the above inequality over $j = 1,\dots, t\ (t\geq 1)$, we have
\begin{equation}\label{eq:sumover}
\sum_{j=1}^{t} \mathbb{E}[\| v_{j} - v_{j-1} \|^2] 
\leq \frac{\eta L}{2 - \eta L} \Big[ \mathbb{E}[ \|v_{0} \|^2] - \mathbb{E}[\| v_{t} \|^2] \Big].  
\end{equation}

By Lemma \ref{lem:var_diff_01}, we have
\begin{align*}
 \mathbb{E}[ \| \nabla F(w_{t}) - v_{t} \|^2 ] & \leq \sum_{j = 1}^{t} \mathbb{E}[ \| v_{j} - v_{j-1} \|^2 ] + \mathbb{E}[\| \nabla F(w_{0}) - v_{0} \|^2] \\ & \overset{\eqref{eq:sumover}}{\leq}  \frac{\eta L}{2 - \eta L} \Big[ \mathbb{E}[ \|v_{0} \|^2] - \mathbb{E}[\| v_{t} \|^2] \Big] + \mathbb{E}[\| \nabla F(w_{0}) - v_{0} \|^2]. \qedhere 
\end{align*}
\end{proof}

\begin{lem}\label{lem_bound_w0_02}
\textit{Suppose that Assumptions \ref{ass_Lsmooth} and \ref{ass_convex} hold. Consider $v_{0}$ defined as \eqref{eq:v_0} in iSARAH (Algorithm \ref{isarah}). Then we have, 
\begin{align*}
& \frac{ \eta L}{2 - \eta L}  \mathbb{E}[ \| v_0 \|^2 ] + \mathbb{E}[ \| \nabla F(w_{0}) - v_{0} \|^2 ] \\ & \qquad \qquad \qquad \leq \frac{2}{2 - \eta L} \left( \frac{4 L \mathbb{E} [F(w_0) - F(w_*)] + 2 \mathbb{E} \left[ \| \nabla f (w_{*}; \xi) \|^2 \right] - \mathbb{E} [\| \nabla F(w_0) \|^2]}{b} \right) \\ &\qquad \qquad \qquad \qquad + \frac{ \eta L}{2 - \eta L}  \mathbb{E}[ \| \nabla F(w_0) \|^2 ]. \tagthis \label{eq_0002}
\end{align*}}
\end{lem}

\begin{proof}
By Corollary \ref{cor_two_lemmas}, we have
\begin{align*}
& \frac{ \eta L}{2 - \eta L}  \mathbb{E}[ \| v_0 \|^2 | w_0 ] - \frac{ \eta L}{2 - \eta L}   \| \nabla F(w_0) \|^2  + \mathbb{E}[ \| \nabla F(w_{0}) - v_{0} \|^2 | w_0 ] \\ 
&\qquad\qquad\qquad  = \frac{2}{2 - \eta L} \Big [ \mathbb{E} [ \| v_0 \|^2 | w_0 ] - \| \nabla F(w_0) \|^2 \Big ] \\
&\qquad\qquad\qquad  = \frac{2}{2 - \eta L} \Big [ \mathbb{E} [ \| v_0 - \nabla F(w_0)\|^2 | w_0 \Big ] \\
&\qquad\qquad\qquad  \overset{\eqref{eq_prob_bound_w0_02}}{\leq} \frac{2}{2 - \eta L} \left( \frac{4 L [F(w_0) - F(w_*)] + 2 \mathbb{E}
 \left[ \| \nabla f (w_{*}; \xi) \|^2 \right] - \| \nabla F(w_0) \|^2}{b} \right).
\end{align*}

Taking the expectation and adding $\frac{ \eta L}{2 - \eta L}  \mathbb{E}[ \| \nabla F(w_0) \|^2 ]$ for both sides, the desired result is achieved. 
\end{proof}

We then derive this basic result for the convex case by using Lemmas \ref{lem_bound_var_diff_str_02} and \ref{lem_bound_w0_02}. 
\begin{lem}\label{lem_keylemma_convex}
\textit{Suppose that Assumptions \ref{ass_Lsmooth} and \ref{ass_convex} hold. Consider  iSARAH-IN (Algorithm \ref{isarah_in}) with $\eta \leq 1/L$. Then, we have 
\begin{align*}
\mathbb{E}[ \| \nabla F(\tilde{w})\|^2 ] 
& \leq \frac{2}{\eta (m + 1)} \mathbb{E} [ F(w_0) - F(w_{*})]  
+ \frac{ \eta L}{2 - \eta L} \mathbb{E} [ \| \nabla F(w_0) \|^2 ] \\ & + \frac{2}{2 - \eta L} \left( \frac{4 L \mathbb{E} [F(w_0) - F(w_*)] + 2 \mathbb{E} \left[ \| \nabla f (w_{*}; \xi) \|^2 \right] - \mathbb{E} [\| \nabla F(w_0) \|^2]}{b} \right), \tagthis \label{eq:agasgsasw}
\end{align*}
where $w_{*}$ is any optimal solution of $F(w)$; and $\xi$ is the random variable. }
\end{lem}

\begin{proof}
By Lemma \ref{lem_bound_var_diff_str_02}, we have
\begin{align*}
\sum_{t=0}^{m} \mathbb{E}[ \| \nabla F(w_{t}) - v_{t} \|^2 ] 
 \leq  \frac{m\eta L}{2 - \eta L} \mathbb{E}[ \|v_{0} \|^2] + (m+1) \mathbb{E}[\| \nabla F(w_{0}) - v_{0} \|^2]. \tagthis \label{eq:abcdef}
\end{align*}
Hence, by Lemma \ref{lem_main_derivation} with $\eta\leq 1/L$, we have
\allowdisplaybreaks
\begin{align*}
\sum_{t=0}^{m} \mathbb{E}[ \| \nabla F(w_{t})\|^2 ] 
 & \leq \frac{2}{\eta} \mathbb{E}[ F(w_{0}) - F(w_{*})] + \sum_{t=0}^{m} \mathbb{E}[ \| \nabla F(w_{t}) - v_{t} \|^2 ] 
 \\
& \overset{\eqref{eq:abcdef}}{\leq} \frac{2}{\eta} \mathbb{E}[ F(w_{0}) - F(w_{*})]  + \frac{m\eta L}{2 - \eta L} \mathbb{E}[ \| v_{0} \|^2 ] \\ & \qquad + (m+1) \mathbb{E}[\| \nabla F(w_{0}) - v_{0} \|^2]. \tagthis\label{eq:thm1conv}
\end{align*}
Since $\tilde{w} = w_t$, where $t$ is picked uniformly at random from $\{0,1,\dots,m\}$. The following holds,
\begin{align*}
& \mathbb{E}[ \| \nabla F(\tilde{w})\|^2 ] = \frac{1}{m+1}\sum_{t=0}^{m} \mathbb{E}[ \| \nabla F(w_{t})\|^2 ]
 \\
&\overset{\eqref{eq:thm1conv}}{\leq} \frac{2}{\eta (m + 1)} \mathbb{E}[ F( w_{0}) - F(w_{*})]   + \frac{ \eta L}{2 - \eta L}  \mathbb{E}[ \| v_0 \|^2 ] + \mathbb{E}[\| \nabla F(w_{0}) - v_{0} \|^2] 
\\
& \overset{\eqref{eq_0002}}{\leq} \frac{2}{\eta (m + 1)} \mathbb{E}[ F(w_0) - F(w_{*})]  
+ \frac{ \eta L}{2 - \eta L}  \mathbb{E}[ \| \nabla F(w_0) \|^2 ] \\ &\qquad + \frac{2}{2 - \eta L} \left( \frac{4 L \mathbb{E} [F(w_0) - F(w_*)] + 2 \mathbb{E} \left[ \| \nabla f (w_{*}; \xi) \|^2 \right] - \mathbb{E} [\| \nabla F(w_0) \|^2]}{b} \right). \qedhere 
\end{align*}
\end{proof}

This expected bound for $\| \nabla F(\tilde{w})\|^2$ will be used for deriving both one-loop and multiple-loop results in the convex case. 

Lemma \ref{lem_keylemma_convex} can be used to get the following result for non-strongly convex. 

\begin{thm}\label{thm:generalconvex_SARAH_IN}
\textit{Suppose that Assumptions \ref{ass_Lsmooth}, \ref{ass_convex}, and \ref{ass_bounded_gradient_solution} hold. Consider  iSARAH-IN (Algorithm \ref{isarah_in}) with $\eta = \frac{1}{L\sqrt{m+1}} \leq \frac{1}{L}$, $b = 2\sqrt{m+1}$ and a given $w_0$. Then we have, 
\begin{align*}
\mathbb{E}[ \| \nabla F(\tilde{w})\|^2 ] \leq \frac{1}{\sqrt{m+1}} \left[ 6 L [ F(w_0) - F(w_{*})]  + 2 \sigma_*^2 \right ]. 
\end{align*}}
\end{thm}

\begin{proof}
By Lemma \ref{lem_keylemma_convex}, for any given $w_0$, we have
\begin{align*}
\mathbb{E}[ \| \nabla F(\tilde{w})\|^2 ] 
& \leq \frac{2}{\eta (m + 1)} [ F(w_0) - F(w_{*})]  
+ \frac{ \eta L}{2 - \eta L}  \| \nabla F(w_0) \|^2 \\ 
&\qquad + \frac{2}{2 - \eta L} \left( \frac{4 L [F(w_0) - F(w_*)] + 2 \mathbb{E} \left[ \| \nabla f (w_{*}; \xi) \|^2 \right] -  \| \nabla F(w_0) \|^2}{b} \right) \\
& \leq \frac{2 L}{\sqrt{m+1}} [ F(w_0) - F(w_{*})]  + \frac{1}{2 - \eta L} \frac{4L}{\sqrt{m+1}} [ F(w_0) - F(w_{*})] \\ & \qquad + \frac{2}{2 - \eta L} \frac{\mathbb{E} \left[ \| \nabla f (w_{*}; \xi) \|^2 \right]}{\sqrt{m+1}} \\
& \leq \frac{2 L}{\sqrt{m+1}} [ F(w_0) - F(w_{*})] \\ & \qquad + \frac{1}{\sqrt{m+1}} \left[ 4 L [ F(w_0) - F(w_{*})]  + 2 \mathbb{E} [ \| \nabla f (w_{*}; \xi) \|^2 ] \right] \\
& \overset{\eqref{eq_bounded_gradient_solution}}{\leq} \frac{1}{\sqrt{m+1}} \left[ 6 L [ F(w_0) - F(w_{*})]  + 2 \sigma_*^2 \right ].
\end{align*}

The second inequality follows since $\eta = \frac{1}{L \sqrt{m+1}}$ and $b = 2 \sqrt{m+1}$. The third inequality follows since $\eta \leq \frac{1}{L}$, which implies $\frac{1}{2 - \eta L} \leq 1$. 
\end{proof}

Based on Theorem~\ref{thm:generalconvex_SARAH_IN}, we are able to derive the following total complexity for iSARAH-IN in the non-strongly convex case.

\begin{cor}\label{cor:generalconvex_1}
\textit{Suppose that Assumptions \ref{ass_Lsmooth}, \ref{ass_convex}, and \ref{ass_bounded_gradient_solution} hold. Consider  iSARAH-IN (Algorithm \ref{isarah_in}) with the learning rate $\eta = \frac{1}{L \sqrt{m+1}}$ and the number of samples $b = 2\sqrt{m+1}$, where $m$ is the total number of iterations, then  $\|\nabla F(\tilde{w})\|^2$ converges
sublinearly in expectation  with a rate of $\Ocal\left(\frac{\max\{ L , \sigma_*^2 \}}{\sqrt{m+1}}\right)$, and therefore, the total complexity to achieve an $\epsilon$-accurate solution is 
\begin{align*}
    \Ocal \left( \frac{\max\{ L , \sigma_*^2 \}}{\epsilon} + \frac{\max\{ L^2 , \sigma_*^4 \}}{\epsilon^2} \right). 
\end{align*}}
\end{cor}

\begin{proof}
It is easy to see that to achieve $\mathbb{E}[ \| \nabla F(\tilde{w})\|^2 ] \leq \epsilon$ we need 
\begin{align*}
    m + 1 = \frac{(6 L [ F(w_0) - F(w_{*})]  + 2 \sigma_*^2)^2}{\epsilon^2},
\end{align*}
and hence the total work is 
\begin{align*}
    b + 2 m &= 2\sqrt{m+1} + 2m = \Ocal \left( \frac{\max\{ L , \sigma_*^2 \}}{\epsilon} + \frac{\max\{ L^2 , \sigma_*^4 \}}{\epsilon^2} \right). \qedhere 
\end{align*}
\end{proof}

\subsubsection{Non-Convex Case}

We now move to the non-convex case. We begin by stating and proving a lemma similar to Lemma \ref{lem_bound_var_diff_str_02}, bounding $\mathbb{E}[ \| \nabla F(w_{t}) - v_{t} \|^2 ]$, but without Assumption 
 \ref{ass_convex}. 
\begin{lem}\label{lem:var_diff_mb_02}
\textit{Suppose that Assumption \ref{ass_Lsmooth} holds. Consider $v_{t}$ defined as \eqref{eq:vt} in iSARAH-IN (Algorithm \ref{isarah_in}). Then for any $t\geq 1$, 
\begin{align*}
\mathbb{E}[ \| \nabla F(w_{t}) - v_{t} \|^2 ]  \leq \mathbb{E}[ \| \nabla F(w_{0}) - v_{0} \|^2 ] + L^2 \eta^2 \sum_{j=1}^{t} \mathbb{E}[\| v_{j-1} \|^2]. \tagthis \label{eq_nc_001}
\end{align*}}
\end{lem}

\begin{proof}
We have, for $t \geq 1$, 
\begin{align*}
\| v_{t} - v_{t-1} \|^2 &\overset{\eqref{eq:vt}}{=} \| \nabla f (w_{t};\xi_t) - \nabla f(w_{t-1};\xi_t) \|^2 
\overset{\eqref{eq:Lsmooth_basic}}{\leq} L^2 \| w_{t} - w_{t-1} \|^2 = L^2 \eta^2 \| v_{t-1} \|^2. \tagthis \label{eq:afsag242}
\end{align*}

Hence, by Lemma \ref{lem:var_diff_01}, 
\begin{align*}
\mathbb{E}[ \| \nabla F(w_{t}) - v_{t} \|^2 ] & \leq \mathbb{E}[\| \nabla F(w_{0}) - v_{0} \|^2] + \sum_{j = 1}^{t} \mathbb{E}[ \| v_{j} - v_{j-1} \|^2 ] \\
&  \overset{\eqref{eq:afsag242}}{\leq} \mathbb{E}[\| \nabla F(w_{0}) - v_{0} \|^2] + L^2 \eta^2 \sum_{j=1}^t \mathbb{E}[ \|v_{j-1}\|^2]. \qedhere 
\end{align*} 
\end{proof}

\begin{lem}\label{lem_nc_lt0}
\textit{Suppose that Assumption \ref{ass_Lsmooth} holds. Consider $v_{t}$ defined as \eqref{eq:vt} in iSARAH-IN (Algorithm \ref{isarah_in}) with $\eta \leq \frac{2}{L(\sqrt{1 + 4m} + 1)}$. Then we have
\begin{align*}
L^2 \eta^2 \sum_{t=0}^{m} \sum_{j=1}^{t} \mathbb{E}[\| v_{j-1} \|^2]  
 - ( 1 - L\eta ) \sum_{t=0}^{m} \mathbb{E} [ \| v_{t} \|^2 ] \leq 0. \tagthis \label{eq_nc_002}
\end{align*}}
\end{lem}

\begin{proof}
For $\eta \leq \frac{2}{L(\sqrt{1 + 4m} + 1)}$, we have
\allowdisplaybreaks
\begin{align*}
& L^2 \eta^2 \sum_{t=0}^{m} \sum_{j=1}^{t} \mathbb{E}[\| v_{j-1} \|^2]  
 - ( 1 - L\eta ) \sum_{t=0}^{m} \mathbb{E} [ \| v_{t} \|^2 ] \\
& \qquad \qquad \qquad \qquad = L^2 \eta^2 \Big[ m  \mathbb{E}\|v_{0} \|^2 + (m-1) \mathbb{E}\|v_{1} \|^2 + \dots + \mathbb{E}\|v_{m-1}\|^2 \Big ] \\ & \qquad \qquad \qquad \qquad \qquad - (1 - L\eta) \Big [ \mathbb{E}\|v_{0} \|^2 + \mathbb{E}\|v_{1} \|^2 + \dots + \mathbb{E}\|v_{m}\|^2  \Big ] \\
& \qquad \qquad \qquad \qquad \leq [L^2\eta^2 m - (1 - L\eta)] \sum_{t=1}^{m} \mathbb{E} [ \| v_{t-1} \|^2 ] \leq 0,
\end{align*}

since $\eta = \frac{2}{L(\sqrt{1 + 4m} + 1)}$ is a root of the equation $L^2\eta^2 m - (1 - L\eta) = 0$. 
\end{proof}

With the help of the above lemmas, we are able to derive our result for non-convex.
\begin{thm}\label{thm:nonconvex_SARAH_IN}
\textit{Suppose that Assumption \ref{ass_Lsmooth} holds and $\mathbb{E} [ \| \nabla f (w_{0}; \xi) \|^2 ]$ is finite. Consider  iSARAH-IN (Algorithm \ref{isarah_in}) with $\eta \leq \frac{2}{L(\sqrt{1 + 4m} + 1)} \leq \frac{1}{L}$, $b = \sqrt{m+1}$ and a given $w_0$. Then we have, 
\begin{align*}
\mathbb{E}[ \| \nabla F(\tilde{w})\|^2 ] \leq \frac{2}{\eta(m+1)} [F(w_0) - F^*] + \frac{1}{\sqrt{m+1}} \Big( \mathbb{E}[ \| \nabla f (w_{0} ; \xi) \|^2 ] \Big), \tagthis \label{eq_thm_003}
\end{align*}
where $F^*$ is any lower bound of $F$; and $\xi$ is some random variable.}  
\end{thm}

\begin{proof}
Let $F^*$ be any lower bound of $F$. By Lemma \ref{lem_main_derivation} and since $\tilde{w} = w_t$, where $t$ is picked uniformly at random from $\{0,1,\dots,m\}$, we have
\allowdisplaybreaks
\begin{align*}
& \mathbb{E}[ \| \nabla F(\tilde{w})\|^2 ] = \frac{1}{m+1}\sum_{t=0}^{m} \mathbb{E}[ \| \nabla F(w_{t})\|^2 ] \\ &\qquad \leq \frac{2}{\eta(m+1)} \mathbb{E}[ F(w_{0}) - F^*] \\ & \qquad + \frac{1}{m+1} \left( \sum_{t=0}^{m} \mathbb{E}[ \| \nabla F(w_{t}) - v_{t} \|^2 ]  
 - ( 1 - L\eta ) \sum_{t=0}^{m} \mathbb{E} [ \| v_{t} \|^2 ] \right) \\
 &\qquad \overset{\eqref{eq_nc_001}}{\leq} \frac{2}{\eta(m+1)} \mathbb{E}[ F(w_{0}) - F^*] + \mathbb{E}[ \| \nabla F(w_{0}) - v_{0} \|^2 ] \\
 &\qquad\qquad\qquad + \frac{1}{m+1} \left( L^2 \eta^2 \sum_{t=0}^{m} \sum_{j=1}^{t} \mathbb{E}[\| v_{j-1} \|^2]  
 - ( 1 - L\eta ) \sum_{t=0}^{m} \mathbb{E} [ \| v_{t} \|^2 ] \right) \\
 &\qquad \overset{\eqref{eq_nc_002}}{\leq} \frac{2}{\eta(m+1)} \mathbb{E}[ F(w_{0}) - F^*] + \mathbb{E}[ \| \nabla F(w_{0}) - v_{0} \|^2 ] \\
 &\qquad \overset{\eqref{eq_prob_bound_w0}}{\leq} \frac{2}{\eta(m+1)} \mathbb{E}[ F(w_{0}) - F^*] + \frac{1}{b} \mathbb{E}[ \| \nabla f (w_{0} ; \xi) \|^2 ]. 
\end{align*}

For any given $w_0$ and $b = \sqrt{m+1}$, we could achieve the desired result. 
\end{proof}

Based on Theorem~\ref{thm:nonconvex_SARAH_IN}, we are able to derive the following total complexity for iSARAH-IN in the non-convex case. 

\begin{cor}\label{cor:nonconvex_1}
\textit{Suppose that Assumption \ref{ass_Lsmooth} holds and $\mathbb{E} [ \| \nabla f (w_{0}; \xi) \|^2 ]$ is finite. Consider  iSARAH-IN (Algorithm \ref{isarah_in}) with the learning rate $\eta = \Ocal\left(\frac{1}{L \sqrt{m+1}}\right)$ and the number of samples $b = \sqrt{m+1}$, where $m$ is the total number of iterations, then  $\|\nabla F(\tilde{w})\|^2$ converges
sublinearly in expectation  with a rate of $\Ocal\left(\sqrt{\frac{1}{m+1}}\right)$, and therefore, the total complexity to achieve an $\epsilon$-accurate solution is $\Ocal(1/\epsilon^2)$.   }
\end{cor}

\begin{proof}
Same as non-strongly convex case, since $\mathbb{E} [ \| \nabla f (w_{0}; \xi) \|^2 ]$ is finite, to achieve $\mathbb{E}[ \| \nabla F(\tilde{w})\|^2 ] \leq \epsilon$ we need $m = \Ocal(1/\epsilon^2)$ and hence the total work is $\sqrt{m} + 2m = \Ocal\left(\frac{1}{\epsilon} + \frac{1}{\epsilon^2} \right) = \Ocal\left(\frac{1}{\epsilon^2} \right)$. 
\end{proof}

\subsection{Multiple-loop iSARAH Results}

In this section, we analyze multiple-loop results of Inexact SARAH (Algorithm \ref{isarah}).

\subsubsection{Strongly Convex Case}

We now turn to the discussion on the convergence of iSARAH under the strong convexity assumption on $F$. 

\begin{thm}\label{thm:multiple_loop_stronglyconvex_01}
\textit{Suppose that Assumptions \ref{ass_Lsmooth}, \ref{ass_stronglyconvex}, \ref{ass_convex}, and \ref{ass_bounded_gradient_solution} hold. Consider iSARAH (Algorithm \ref{isarah}) with the choice of $\eta$, $m$, and $b$ such that
\begin{align*}
\alpha &= \frac{1}{\mu \eta (m + 1)} + \frac{ \eta L}{2 - \eta L} + \frac{4 \kappa - 2}{b (2 - \eta L)} < 1.  
\end{align*}
(Note that $\kappa = L/\mu$.) Then, we have
\begin{align*}
\mathbb{E}[ \| \nabla F(\tilde{w}_s)\|^2 ] - \Delta  \leq \alpha^s ( \| \nabla F(\tilde{w}_0)\|^2 - \Delta), \tagthis \label{eq_multiple_loop_stronglyconvex_01}
\end{align*}
where
\begin{align*}
\Delta = \frac{\delta}{1- \alpha} \ \text{and} \ \delta = \frac{4}{b(2 - \eta L)}  \sigma_*^2.
\end{align*}}
\end{thm}

\begin{proof}
By Lemma \ref{lem_keylemma_convex}, with $\tilde{w} = \tilde{w}_s$ and $w_0 = \tilde{w}_{s-1}$, we have
\allowdisplaybreaks
\begin{align*}
& \mathbb{E}[ \| \nabla F(\tilde{w}_s)\|^2 ] \\ 
& \qquad \leq \frac{2}{\eta (m + 1)} \mathbb{E} [ F(\tilde{w}_{s-1}) - F(w_{*})]  
+ \frac{ \eta L}{2 - \eta L} \mathbb{E} [ \| \nabla F(\tilde{w}_{s-1}) \|^2 ] \\ &\qquad + \frac{2}{2 - \eta L} \left( \frac{4 L \mathbb{E} [F(\tilde{w}_{s-1}) - F(w_*)] + 2 \mathbb{E} \left[ \| \nabla f (w_{*}; \xi) \|^2 \right] - \mathbb{E} [\| \nabla F(\tilde{w}_{s-1}) \|^2]}{b} \right) \\
& \qquad \overset{\eqref{eq:strongconvexity2}}{\leq} \left( \frac{1}{\mu \eta (m + 1)} + \frac{ \eta L}{2 - \eta L} + \frac{4 \kappa - 2}{b (2 - \eta L)} \right) \mathbb{E} [ \| \nabla F(\tilde{w}_{s-1}) \|^2 ] \\
& \qquad \qquad + \frac{4}{b(2 - \eta L)}  \mathbb{E} \left[ \| \nabla f (w_{*}; \xi) \|^2 \right] \\
& \qquad \overset{\eqref{eq_bounded_gradient_solution}}{\leq} \left( \frac{1}{\mu \eta (m + 1)} + \frac{ \eta L}{2 - \eta L} + \frac{4 \kappa - 2}{b (2 - \eta L)} \right) \mathbb{E} [ \| \nabla F(\tilde{w}_{s-1}) \|^2 ] \\
& \qquad + \frac{4}{b(2 - \eta L)}  \sigma_*^2 \tagthis \label{eq_multiple_loop_stronglyconvex_0001} \\
& \qquad = \alpha \mathbb{E} [ \| \nabla F(\tilde{w}_{s-1}) \|^2 ] + \delta \\
& \qquad \leq \alpha^s \| \nabla F(\tilde{w}_{0})\|^2 + \alpha^{s-1} \delta +  \dots + \alpha \delta + \delta  \\
& \qquad \leq \alpha^s \| \nabla F(\tilde{w}_{0})\|^2 + \delta \frac{1 - \alpha^s}{1 - \alpha} \\
& \qquad = \alpha^s  \| \nabla F(\tilde{w}_0)\|^2 + \Delta(1 - \alpha^s) \\
& \qquad = \alpha^s ( \| \nabla F(\tilde{w}_0)\|^2 - \Delta) + \Delta. 
\end{align*}
By adding $-\Delta$ to both sides, we achieve the desired result. 
\end{proof}

Based on Theorem~\ref{thm:multiple_loop_stronglyconvex_01}, we are able to derive the following total complexity for iSARAH in the strongly convex case. 

\begin{cor}\label{cor:multiple_loop_stronglyconvex_01}
\textit{Let $\eta = \frac{2}{5 L}$, $m = 20\kappa - 1$, and $b = \max\left\{ 20\kappa - 10, \frac{20 \sigma_*^2}{\epsilon}\right\}$ in Theorem \ref{thm:multiple_loop_stronglyconvex_01}. Then, the total work complexity to achieve $\mathbb{E}[ \| \nabla F(\tilde{w}_s)\|^2 ] \leq \epsilon$ is $\Ocal\left(  \max\left\{  \frac{\sigma_*^2}{\epsilon}, \kappa \right\} \log\left( \frac{1}{\epsilon} \right)   \right)$. }
\end{cor}

\begin{proof}
With $\eta = \frac{2}{5 L}$, $m = 20\kappa - 1$, and $b = \max\left\{ 20\kappa - 10, \frac{20 \sigma_*^2}{\epsilon}\right\}$, from \eqref{eq_multiple_loop_stronglyconvex_0001}, we have
\begin{align*}
\mathbb{E}[ \| \nabla F(\tilde{w}_s)\|^2 ] & \leq \left(\frac{1}{8} + \frac{1}{4} + \frac{1}{8} \right) \mathbb{E} [ \| \nabla F(\tilde{w}_{s-1}) \|^2 ] + \frac{\epsilon}{8} \\
& \leq \frac{1}{2} \mathbb{E} [ \| \nabla F(\tilde{w}_{s-1}) \|^2 ] + \frac{\epsilon}{8} \\
& \leq \frac{1}{2^s} \| \nabla F(\tilde{w}_0) \|^2 + \frac{\epsilon}{4}. 
\end{align*}
Since $\mathbb{E} [ \| \nabla f (w_{*}; \xi) \|^2 ]$ is finite, to guarantee that  $\mathbb{E}[ \| \nabla F(\tilde{w}_s)\|^2 ] \leq \epsilon$, it is sufficient to make $\frac{1}{2^s} \| \nabla F(\tilde{w}_0) \|^2 = \frac{3}{4}\epsilon$ or equivalently $s = \log\left(\frac{\| \nabla F(\tilde{w}_0) \|^2}{\frac{3}{4} \epsilon} \right)$. This implies the total complexity to achieve an $\epsilon$-accuracy solution is $(b + m)s = \Ocal\left( \left( \max\left\{  \frac{\sigma_*^2}{\epsilon}, \kappa \right\} + \kappa  \right) \log\left( \frac{1}{\epsilon} \right)   \right) = \Ocal\left(  \max\left\{  \frac{\sigma_*^2}{\epsilon}, \kappa \right\} \log\left( \frac{1}{\epsilon} \right)   \right)$. 
\end{proof}

\subsubsection{Non-Strongly Convex Case}

We turn to the analysis of the convergence rate of the multiple-loop iSARAH in the non-strongly convex case. As mentioned in the introduction, we are able to achieve best sample complexity rates of any stochastic algorithm, to the best of our knowledge, but under an additional reasonably mild assumption. We introduce this assumption below.  

\begin{ass}\label{ass_generalconvex_01}
\textit{Let $\tilde{w}_0$,$\dots$,$\tilde{w}_{\cal T}$ be the (outer) iterations of Algorithm \ref{isarah}. We assume that there exist $M > 0$ and $N > 0$ such that, $k = 0,\dots, {\cal T}$,
\begin{align*}
F(\tilde{w}_k) - F(w_*) \leq M \| \nabla F(\tilde{w}_k) \|^2 + N,  \tagthis \label{eq_ass_generalconvex_01}
\end{align*}
where $F(w_*)$ is the optimal value of $F$. }
\end{ass}

Let us discuss Assumption \ref{ass_generalconvex_01}.  First, we note that the assumption only requires to hold for the outer iterations $\tilde{w}_0$,$\dots$,$\tilde{w}_{\cal T}$ of Algorithm \ref{isarah} instead of holding for all $w \in \mathbb{R}^d$ or for all of the inner iterates. Moreover, this assumption is clearly weaker than the Polyak-Lojasiewicz (PL) condition, which has been studied and discussed in \cite{Polyak1964,nesterov2006cubic,polyak_condition} which itself is weaker than strong convexity assumption.  Under PL condition, we simply have $N=0$ in \eqref{eq_ass_generalconvex_01} and as we will discuss below we can recover better convergence rate in this case. 
On the other hand, if PL condition does not hold but if the sequence of iterates $\{\tilde{w}_k\}$ remains in a  set, say ${\cal W}$, on which the objective function is bounded from above, that is
for all $w\in {\cal W}$, $F(w)\leq F_{\max}$ for some finite value $F_{max}$, then Assumption \ref{ass_generalconvex_01}
is satisfied with $N= F_{\max}-F(w_*)$ and $M=0$, where $F(w_*)$ is the optimal value of $F$. In other words, Assumption \ref{ass_generalconvex_01} is a relaxation of the boundedness assumption and the PL condition. As an example, consider the following modification of the logistic function (for some $\lambda >0$)
\begin{align*}
F(w)= \left \{\begin{array}{lr} \log (1+e^{-w}) & w\geq -2 \\
 \log (1+e^{-w})+ \frac{\lambda}{2} (w+2)^2& w<-2 \end{array}\right.
\end{align*}
which is not strongly convex and does not satisfy the PL condition.  This function can be considered in practice instead of the usual logistic function as it simply adds somewhat larger penalty on solutions that are far away from the optimal. Notice that this function is continuous and continuously differentiable. When, $w < -2$, the function is strongly convex and satisfies the PL condition. When $w \geq -2$, the function is bounded in $(0, \log(1 + e^2)]$. Therefore, it satisfies Assumption~\ref{ass_generalconvex_01}. This is a simplified example to make derivations easy. However, it is possible to generalize this example by  modifying standard logistic  loss by adding strongly convex penalty outside some ball containing the optimal solution. 

\begin{thm}\label{thm:generalconvex_SARAH_FULL}
\textit{Suppose that Assumptions \ref{ass_Lsmooth}, \ref{ass_convex}, \ref{ass_bounded_gradient_solution}, and \ref{ass_generalconvex_01} hold. Consider iSARAH (Algorithm \ref{isarah}) with the choice of $\eta$, $m$, and $b$ such that
\begin{align*}
\alpha_c &= \frac{2 M}{\eta (m+1)} + \frac{ \eta L}{2 - \eta L} + \frac{8 L M - 1}{b(2 - \eta L)} < 1.  
\end{align*}
Then, we have
\begin{align*}
\mathbb{E}[ \| \nabla F(\tilde{w}_s)\|^2 ] - \Delta_c  \leq \alpha_c^s ( \| \nabla F(\tilde{w}_0)\|^2 - \Delta_c), \tagthis \label{eq_multiple_loop_generalconvex_01}
\end{align*}
where
\begin{align*}
\Delta_c = \frac{\delta_c}{1- \alpha_c} \ \text{and} \ \delta_c = \frac{2 N}{\eta (m+1)} + \frac{8 L N}{b(2 - \eta L)} + \frac{4 \sigma_*^2}{b(2 - \eta L)}.
\end{align*}}
\end{thm}

\begin{proof}
By Lemma \ref{lem_keylemma_convex}, with $\tilde{w} = \tilde{w}_s$ and $w_0 = \tilde{w}_{s-1}$, we have
\allowdisplaybreaks
\begin{align*}
& \mathbb{E}[ \| \nabla F(\tilde{w}_s)\|^2 ] \\
& \qquad \leq \frac{2}{\eta (m + 1)} \mathbb{E} [ F(\tilde{w}_{s-1}) - F(w_{*})]  
+ \frac{ \eta L}{2 - \eta L} \mathbb{E} [ \| \nabla F(\tilde{w}_{s-1}) \|^2 ] \\ &\qquad + \frac{2}{2 - \eta L} \left( \frac{4 L \mathbb{E} [F(\tilde{w}_{s-1}) - F(w_*)] + 2 \mathbb{E} \left[ \| \nabla f (w_{*}; \xi) \|^2 \right] - \mathbb{E} [\| \nabla F(\tilde{w}_{s-1}) \|^2]}{b} \right) \\
& \qquad \overset{\eqref{eq_ass_generalconvex_01}}{\leq} \left( \frac{2 M}{\eta (m+1)} + \frac{ \eta L}{2 - \eta L} + \frac{8 L M - 1}{b(2 - \eta L)} \right) \mathbb{E} [ \| \nabla F(\tilde{w}_{s-1}) \|^2 ] \\
& \qquad + \frac{2 N}{\eta (m+1)} + \frac{8 L N}{b(2 - \eta L)} + \frac{4 \mathbb{E} \left[ \| \nabla f (w_{*}; \xi) \|^2 \right]}{b(2 - \eta L)} \\
& \qquad \overset{\eqref{eq_bounded_gradient_solution}}{\leq} \left( \frac{2 M}{\eta (m+1)} + \frac{ \eta L}{2 - \eta L} + \frac{8 L M - 1}{b(2 - \eta L)} \right) \mathbb{E} [ \| \nabla F(\tilde{w}_{s-1}) \|^2 ] \\
& \qquad + \frac{2 N}{\eta (m+1)} + \frac{8 L N}{b(2 - \eta L)} + \frac{4 \sigma_*^2}{b(2 - \eta L)} \\
& \qquad= \alpha_c \mathbb{E} [ \| \nabla F(\tilde{w}_{s-1}) \|^2 ] + \delta_c \\
& \qquad \leq \alpha_c^s ( \| \nabla F(\tilde{w}_0)\|^2 - \Delta_c) + \Delta_c. \qedhere 
\end{align*}
\end{proof}


Now choosing the appropriate values for $\eta$, $m $, and $b$, we can follow the proof of Corollary~\ref{cor:multiple_loop_stronglyconvex_01} to achieve the following complexity result. 
\begin{cor}\label{cor:multiple_loop_generalconvex_01}
\textit{In Theorem~\ref{thm:generalconvex_SARAH_FULL}, let $\eta = \frac{2}{5 L}$, $m = \max\left\{  40 L M - 1 , \frac{120 L N}{ \epsilon} - 1 \right \}$, and $b = \max\left \{ 40 L M - 5 , \frac{120 L N}{\epsilon} , \frac{60 \sigma_*^2}{\epsilon} \right \}$. Then, the total work complexity to achieve $\mathbb{E}[ \| \nabla F(\tilde{w}_s)\|^2 ] \leq \epsilon$ is 
\begin{align*}
    (b + m)s = \Ocal\left(  \max\left\{ LM , \frac{\max\{ L N , \sigma_*^2 \}}{\epsilon}  \right \} \log \left( \frac{1}{\epsilon}  \right)  \right).
\end{align*}}
\end{cor}

We can observe that, with the help of Assumption \ref{ass_generalconvex_01}, iSARAH could achieve the best known complexity among stochastic methods (those which do not have access to exact gradient computation) in the non-strongly convex case. 

\section{Conclusion}\label{sec_conclusion}

We have provided the analysis of the inexact version of SARAH, which requires only stochastic gradient information computed on a mini-batch of sufficient size. We provide the one-loop results (iSARAH-IN) in the non-strongly convex and non-convex cases. Moreover, we analyze the multiple-loop results (iSARAH) in the strongly convex case and with an additional assumption (Assumption~\ref{ass_generalconvex_01}) in the non-strongly convex case. With this Assumption \ref{ass_generalconvex_01},  which we argue is reasonable, iSARAH achieves the best known complexity among stochastic methods.

\section*{Acknowledgement}
Katya Scheinberg and Martin Tak\'a\v{c} were partially supported by the U.S. National Science Foundation, under award number  CCF-1618717 and  CCF-1740796.

\bibliographystyle{tfs}
\bibliography{reference}

\end{document}